\documentclass[12pt]{amsart}
\usepackage[a4paper]{geometry}
\usepackage{amsmath,amsthm}
\usepackage{amsfonts,amstext,amssymb, mathtools, comment, dsfont, wasysym}
\usepackage{graphicx}
\usepackage{caption}
\usepackage{subcaption,wrapfig}
\RequirePackage[colorlinks,citecolor=blue,urlcolor=blue]{hyperref}
 \usepackage[usenames,dvipsnames]{pstricks}

\newcommand{\Prob}{{\mathbb P}}

\newcommand{\D}{{\mathrm d}}
\newcommand{\Expe}{{\mathrm e}}

\renewcommand{\a}{\alpha}

\newcommand{\R}{\mathbb R}
\newcommand{\ceil}[1]{\lceil{#1}\rceil}

\newcommand{\RV}{\operatorname{RV}}
\newcommand{\floor}[1]{\lfloor{#1}\rfloor}

\newtheorem{theorem}{Theorem}[section]
\newtheorem{proposition}[theorem]{Proposition}

\newtheorem{lemma}[theorem]{Lemma}
\newtheorem{remark}[theorem]{Remark}

\theoremstyle{definition}
\newtheorem{condition}[theorem]{Condition}

\numberwithin{equation}{section}

\newcommand\convdistr{
  \xrightarrow{\:\scriptscriptstyle\smash{\mathcal{D}}\:}}
  \newcommand{\e}{\mathrm{e}} 
\newcommand\eqdistr{\stackrel{\:\scriptscriptstyle\smash{\mathcal{D}}\:}{=}}
\newcommand{\oh}{{\mathrm{o}}} 
\newcommand{\Oh}{{\mathrm{O}}} 

\newcommand{\abs}[1]{\lvert{#1}\rvert}

\renewcommand{\le}{\leqslant}
\renewcommand{\ge}{\geqslant}
\renewcommand{\leq}{\leqslant}
\renewcommand{\geq}{\geqslant}

\usepackage[dvipsnames]{xcolor}

\definecolor{asparagus}{rgb}{0.53, 0.66, 0.42}
\definecolor{olive}{rgb}{0.5, 0.5, 0.0}
\definecolor{antiquefuchsia}{rgb}{0.57, 0.36, 0.51}
\definecolor{golden(brown)}{rgb}{0.6, 0.4, 0.08} %
\definecolor{golden(darkbrown)}{rgb}{0.48, 0.32, 0.064} %
\definecolor{gray-asparagus}{rgb}{0.27, 0.35, 0.27} 
\definecolor{glaucous}{rgb}{0.38, 0.51, 0.71}
\definecolor{airforceblue}{rgb}{0.36, 0.54, 0.66}
\definecolor{blue(munsell)}{rgb}{0.0, 0.5, 0.69}
\definecolor{burntorange}{rgb}{0.8, 0.33, 0.0}

\title{On the longest gap between power-rate arrivals}

\author[S. Asmussen]{S\o{}ren Asmussen}
\address{Aarhus University, Department of Mathematical Sciences, Ny Munkegade, DK-8000 Aarhus C, Denmark}
\email{asmus@math.au.dk}

\author[J. Ivanovs]{Jevgenijs Ivanovs}
\address{Aarhus University, Department of Mathematical Sciences, Ny Munkegade, DK-8000 Aarhus C, Denmark}
\email{jevgenijs.ivanovs@math.au.dk}

\author[J. Segers]{Johan Segers}
\address{Universit\'{e} catholique de Louvain, Institut de Statistique, Biostatistique et Sciences Actuarielles, Voie du Roman Pays~20, B-1348 Louvain-la-Neuve, Belgium}
\email{johan.segers@uclouvain.be}

\date{\today}


\begin{document}
\maketitle

\begin{abstract}
Let $L_t$ be the longest gap before time $t$ in an inhomogeneous Poisson process
with rate function $\lambda_t$ proportional to $t^{\alpha-1}$ for some $\alpha\in(0,1)$. It is shown that
$\lambda_tL_t-b_t$ has a limiting Gumbel distribution for  suitable constants $b_t$ and that
the distance of this  longest gap from $t$ is   asymptotically of the form $(t/\log t)E$ for an exponential
random variable $E$. The analysis is performed via weak convergence of related point processes. 
Subject to a 
weak technical condition, the results
are extended to include a slowly varying term in $\lambda_t$.
\end{abstract}

\section{Introduction and main results}
\label{S:Intr}

Let $(\mathcal N_t)_{t\geq 0}$ be an inhomogeneous Poisson process with rate $\lambda_t$ such that $\Lambda(t) = \int_0^t \lambda_s \, \D s < \infty$ for all $t>0$. The epochs of $\mathcal N$, in increasing order, are denoted by $T_i$, $i=1,2,\ldots$, so that the gaps are given by $R_i=T_i-T_{i-1}$ with~$T_0=0$. The objects of study of the present paper are the longest gap, $L_t$, before time $t$ and its right-end position, $\sigma_t$: 
\begin{align}
\label{eq:Lt}
  L_t 
  &= 
  \max_{i\geq 1}\{R_i:T_i\leq t\}, \\
\label{eq:sigmat}
  \sigma_t 
  &=
  \min_{i\geq 1}\{T_i : R_i = L_t\}.
\end{align}
Note that the definition does not include the gap straddling time $t$, but this is in fact
unimportant for our
asymptotic results, see Remark~\ref{rem:count-the-gap}.
 
In the homogeneous case, the discrete time analogue of the longest gap 
is the longest run, $L_n$, of ones 
before time~$n$ in a Bernoulli$(p)$
sequence.   The study of the longest run has a long history going back to, among others,
\cite{erdos, vonMises}; a recent survey is in \cite{balakrishnan2011runs}.
A main result is that $L_n$ is of order $\log_{1/p} n$. In the homogeneous Poisson case, $\lambda_t \equiv \lambda$, there is a neat analogue of this: 
\begin{equation}\label{18.11a}
\lambda L_t-\log (\lambda t)\,\convdistr\,G\ \text{as }t\to\infty\,,
\end{equation}
where $G$ is Gumbel with cumulative distribution function (cdf) $\Prob(G\le x)=$ $\exp(-\e^{-x})$. The proof is equally neat:
with $M^\pm_t=$ $\max_{i\le \lambda t(1\pm\epsilon)}R_i$ for $\epsilon > 0$ one has $M^-_t\le L_t\le M^+_t$ for large $t$ with high probability. 
Further,  by standard extreme value theory, the random variables $\lambda M^\pm_t-\log \{\lambda t(1\pm\epsilon)\}$ 
 have Gumbel limits as $t \to \infty$, so one can just let first $t$ tend to infinity and next
$\epsilon$ tend to $0$. We provide some further comments and references in Remark~\ref{rem:homogeneous} below.

As mentioned above, our interest is in time inhomogeneity. This may occur in at least two ways.
Firstly, one may consider fluctuations around a long-term average which is conveniently modelled in a hidden Markov setting,
see \cite{antzoulakos1999waiting,Olebook,fu1994distribution}. Secondly, the rates $\lambda_t$ may exhibit a
systematic deterministic trend. The only reference here seems to be~\cite{AIRN} (though cf.\ also~\cite{Karlin}),
continuing a study of \cite{A5} related to problems from computer reliability. The results in~\cite{AIRN}
are of large deviations type, giving asymptotic estimates of $\Prob(L_t<\ell)$ in the rare-event
setting where $t\to\infty$ with $\ell$ fixed.  Our concern here is the typical behaviour, that is, analogues
of~\eqref{18.11a}.

As in~\cite{AIRN}, the quantitative form of $\lambda_t$ is crucial both for the form of the results and
the difficulty of the analysis. First, we concentrate on what is maybe the simplest form, a power function $\lambda_t = \lambda_1 t^{\alpha-1}$, and then provide extensions to regularly varying functions. The power function is a rather natural choice with which to start the analysis, and already this case presents 
substantial challenges. The case $\alpha=1$ is settled by \eqref{18.11a} and the behaviour when
$\alpha>1$ or $\alpha\le 0$  is easily resolved, see Remark~\ref{Rem:18.11a} below.
Thus what is left for analysis is the case $0<\alpha<1$, and here our result
is the following:

\begin{theorem}
\label{thm:main}
Let $({\mathcal N}_t)_{t \ge 0}$ be an inhomogeneous Poisson process with rate $\lambda_t = \lambda_1 t^{\a-1}$ with $\lambda_1 > 0$ and $\a\in(0,1)$. For $L_t$ and $\sigma_t$ as in \eqref{eq:Lt} and \eqref{eq:sigmat}, we have
\[
  \left(\lambda_t L_t-b_t, \, \frac{t-\sigma_t}{t}\log t\right)
  \convdistr 
  \left(G,E_{\a(1-\a)}\right)
  \qquad \text{ as }t\to\infty,
\]
where $b_t=\a\log t-\log\log t-\log(\a(1-\a)/\lambda_1)$ and $G,E_{\a(1-\a)}$ are independent random variables: $G$ is Gumbel and $E_{\a(1-\a)}$ is exponential with rate $\a(1-\a)$.
\end{theorem}

In fact, we prove a much more general result establishing weak convergence of a sequence of point processes, from which Theorem~\ref{thm:main} easily follows.
Here and as usual, convergence in distribution  of point processes is with respect to the vague topology in the space of Radon measures on~$(-\infty,\infty]^2$.
\begin{theorem}\label{thm:point_proc}
Under the assumptions of Theorem~\ref{thm:main} consider the point process~$\xi_t$ on $(-\infty,\infty]^2$ consisting of the points \[\left(\lambda_t R_i-b_t,\frac{t-T_i}{t}\log t\right)\qquad i=1,2,\ldots\]
Then $\xi_t\convdistr \xi$ as $t\rightarrow\infty$, where $\xi$ is a Poisson point process with intensity measure
\[\mu(\D x,\D z)=\e^{-x}\D x\times \a(1-\a)\e^{-\a(1-\a) z}\D z\,.\]
\end{theorem}

Importantly, in Theorem~\ref{thm:main} we consider the compactified Euclidean plane $(-\infty,\infty]^2$ so that the set $[x,\infty]\times [-z,\infty]$ is compact. The points of $\xi_t$ in this set are affine transformations of couples $(R_i,T_i)$ such that $R_i\geq (x+b_t)/\lambda_t$ and $T_i\leq t(1+z/\log t)$. 
Hence our result concerns all large enough gaps of~$\mathcal N$ up to the time $t+\Oh(t/\log t)$. Furthermore, since vague convergence of point measures implies convergence of the respective points in any compact set~\cite[Prop.\ 3.13]{resnick}, we conclude that the map 
\[\sum_i\delta_{(x_i,z_i)}\mapsto (x,z), \qquad x=\max\{x_i:z_i\geq 0\},\,z=\max\{z_i:x_i=x\}\] is continuous apart from possible discontinuities at point measures with $x_i=x_j$ or $z_i=0$ for some $i\neq j$. Since $\xi$ is not of such form a.s.,  the continuous mapping theorem 
gives that $(\lambda_t L_t-b_t,(1-\sigma_t/t)\log t)\convdistr (X,Z)$, where $(X,Z)$ has the distribution arising from the application of the above map to~$\xi$. A standard calculation reveals that for $z>0$ we have 
\begin{align}
\label{eq:XZ}
  \Prob(X\in\D x,Z\in\D z)
  &=\Prob(\xi(\D x\times\D z)=1, \, \xi((x,\infty)\times(0,\infty))=0)\\
\nonumber
  &= \mu(\D x,\D z)\exp\{-\mu((x,\infty)\times(0,\infty))\} \\
\nonumber
  &= \mu(\D x,\D z)\exp(-\e^{-x}),
\end{align}
proving Theorem~\ref{thm:main}; see also the light-gray region in Figure~\ref{fig:points}.

\begin{remark}\label{rem:count-the-gap}\rm
In order to give a feeling for some further results we consider the first gap exceeding~$L_t$ and its time of occurrence:
$(L_t^+,\sigma^+_t)=(R_{i^+_t},T_{i^+_t})$, where $i^+_t=\min\{i\geq 1:R_i>L_t\}$ is the corresponding index.
From Theorem~\ref{thm:point_proc} and the continuous mapping theorem applied to the appropriate map, we find that 
\[
  \left(\lambda_t L_t-b_t,\, \lambda_t L^+_t-b_t, \, \frac{t-\sigma_t}{t}\log t,\, \frac{\sigma^+_t-t}{t}\log t\right)
  \convdistr 
  \left(X,X^+,Z,Z^+\right),
\]
where the conditional distribution of $X^+,Z^+$ is easily  identified to be
\begin{align}
\label{eq:XZ:plus}
  \Prob(X^+\in \D x^+&,Z^+\in \D z^+ \mid X=x,Z=z)\\
\nonumber
  &=\Prob(\xi(\D x^+\times(-\D z^+))=1,\xi((x,\infty)\times(-z^+,0))=0)\\
\nonumber
  &=
  \mu(\D x^+,-\D z^+)\exp(-\e^{-x}(\e^{\a(1-\a)z^+}-1))
\end{align}
for $x^+>x$ and $z^+>0$; see the dark-grey region in Figure~\ref{fig:points}.
\begin{figure}
\includegraphics[width=0.4\textwidth]{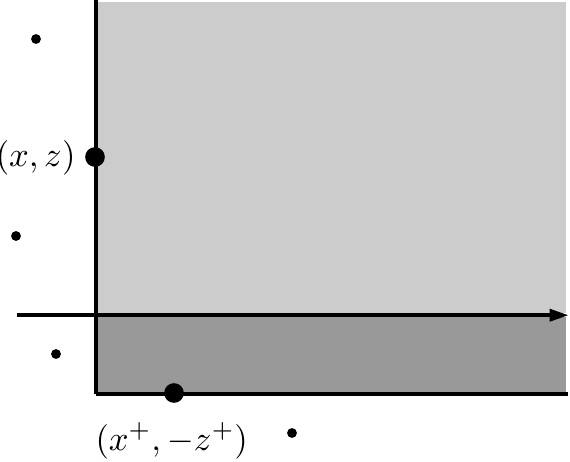}
\caption{The points $(x, z)$ and $(x^+, -z^+)$ and the associated empty regions $(x, \infty) \times (0, \infty)$ in~\eqref{eq:XZ} (light-gray) and $(x, \infty) \times (-z^+, 0)$ in~\eqref{eq:XZ:plus} (dark-gray), respectively.}
\label{fig:points}
\end{figure}
 In particular, we find after some computation that  
 $Z^+\eqdistr E_{\a(1-\a)}\eqdistr Z$. 
One may proceed even further and obtain convergence of extremal processes (on the Skorokhod space of two-sided paths) identifying the record gaps and their times, see~\cite[Prop.\ 4.20]{resnick} for the classical setting.

Finally, note that $L_t^+=\oh_p(t/\log t)$ and so $\Prob(\sigma_t^+-L_t^+>t)\to 1$, showing that the corresponding gap does not straddle time~$t$ in the limit.
\end{remark}

When trying to adapt the above proof of \eqref{18.11a}, with scale constant $\lambda_1 = \a$ say, one quite easily gets $\mathcal N_t\approx t^\a$, which gives a rough estimate of $L_t$ in terms of $\max_{i < t^\a} R_i$. The difficulty is that these interarrival times $R_i$ are no longer
independent nor exponentially distributed. Nevertheless, the $R_i$ are not too far from exponential random variables with rates $\lambda_{T_i} \approx \a \, i^{(\a-1)/\alpha}$, because $T_i\approx i^{1/\alpha}$ for large~$i$. Hence our first step is to consider extreme value theory for sequences 
of i.i.d.\ random variables equipped with weights. Some references in that direction are 
\cite{Kostya99,Gouet15,DeHaan87,Smith88} and, of particular relevance for us,
\cite[Thm.\ 4.1]{WN10}, from which the following result can be extracted:

\begin{proposition}\label{prop:WN}
\label{prop:max}
Let $X_1, X_2, \ldots$ be independent unit exponential random variables and let $\gamma \in (0, \infty)$. 
Then with $M_n=\max_{i=1,\ldots,n}\{i^{\gamma}X_i\}$ we have
\[
  \frac{M_n}{n^{\gamma}} - \beta_n 
  \convdistr G, \qquad \text{ as }n \to \infty,
\]
where $\beta_n = \log (n/\gamma) - \log \log n$ and $G$ is a Gumbel random variable.
\end{proposition}

Our analysis supplements this result by identifying the location of the maximum and providing the analogue of Theorem~\ref{thm:point_proc}. This location is trivially uniform for i.i.d.\ sequences or homogeneous Poisson processes, but has an interesting limiting distribution in the nonhomogeneous case. We also give an extension to
weights in Proposition~\ref{prop:WN} and rates in Theorem~\ref{thm:main} 
which are regularly varying  rather than of simple
power form. Such an extension is of course expected, but the proof is surprisingly complicated,
and in fact, we need some regularity conditions on the slowly varying function.

\begin{remark}\label{rem:homogeneous}\rm
	Despite its simplicity, \eqref{18.11a}  does not seem to have been formulated in the longest run/gap literature. Note that its analogue fails in the Bernoulli setting, because the extreme value behaviour
	of geometric random variables is more complicated than the one of exponential random variables, cf.~\cite[pp.\,24--25]{LLR}.
	
	However, as pointed out by an associate editor and a referee, there are a number of related results in the stochastic geometry literature. Most of these are more general and go deeper, but \eqref{18.11a} can be deduced after some reformulation.
	For example, consider the probability of full coverage of the interval $[0,1]$ in the Boolean model~\cite{hall} on $\R$ with deterministic segments of length $r^{(t)}=(x+\log(\lambda t))/(\lambda t)$ arriving at the rate $\lambda^{(t)}=\lambda t$. By rescaling time we find that  
	\begin{align*}
	\Prob(\lambda L_t-\log(\lambda t)\leq x)
	&=\Prob(L_t/t\leq r^{(t)}) \\
	&=\Prob^{(t)}(\text{full coverage of }[0,1])+\oh(1)\quad\text{as }t\to \infty,
	\end{align*}
	which converges to $\exp(-\Expe^{-x})$ according to~\cite[Thm.\ 2.5]{hall}. For related results in the nonuniform setting see~\cite{hall_nonuniform,husler} and~\cite{molchanov} for more recent work.
	Furthermore, \eqref{18.11a} also follows from~\cite[(2c)]{calka} specifying the limit behaviour of the maximal circumscribed radius of a Poisson--Voronoi tessellation.
	\end{remark}

\begin{remark}\label{Rem:18.11a}\rm
When $\alpha>1$,~\cite{AIRN} gives that the increasing process $L_t$ has a proper limiting distribution, of $L_\infty$, say. 
That is, from~(7) in \cite{AIRN} it follows that $\Prob(L_\infty\geq \ell)\to 0$ as $\ell\to \infty$.
The case $\alpha<0$
is trivial since then $\int_1^\infty\lambda_t\,\D t<\infty$, so that the number of epochs in $[1,\infty)$  is finite with probability~1. 
The boundary case $\alpha=0$ is also easy: if $\lambda_t = \lambda_1 / t$ for some scale constant $\lambda_1 > 0$, then
\begin{equation}\label{18.11b}
\left(\frac{L_t}{t},\frac{\sigma_t}{t}\right)\stackrel{\mathcal{D}}{=}(L_1,\sigma_1).
\end{equation}
Indeed, fix $t > 0$ and define the time-changed process $\mathcal{N}'$ by $\mathcal N'_x=\mathcal N_{tx}$ for $x \ge 0$. Its intensity measure, $\Lambda'$, satisfies $\Lambda'(x,y)=\Lambda(tx,ty)=\lambda_1 \log ((ty)/(tx)) = \Lambda(x,y)$ for any $0<x<y$. It follows that ${\mathcal N}'$ has the same distribution as ${\mathcal N}$, and it is then clear that $(L_{t}/t,\sigma_t/t)$ has the same distribution as $(L_1,\sigma_1)$. 

Finally, observe that $t-\sigma_t$ is of order~$t$ when $\alpha=0$ or $1$, the two boundary cases in Theorem~\ref{thm:main}.
In contrast, Theorem~\ref{thm:main} gives the smaller order $t/\log t$ when $\a\in(0,1)$. 
Therefore, it is intuitive that the limiting random variable $E_{\a(1-\a)}$ must increase to 
$\infty$ as $\a$ approaches~0 or~1. This is indeed the case.
\end{remark}


\section{Weighted exponentials}
\label{sec:exponential}

As in Proposition~\ref{prop:WN}, we consider a sequence $X_1, X_2, \ldots$ of independent, unit exponential random variables. We fix $\gamma > 0$ and let
\begin{equation}
\label{eq:MnX}
  M_n = \max_{i=1,\ldots,n} \{ i^{\gamma} X_i \} 
  \qquad \text{and} \qquad
  \tau_n = \min \{ i = 1, \ldots, n : i^{\gamma} X_i = M_n \},
\end{equation}
denote the partial maximum of the weighted sequence $(i^\gamma X_i)_{i \ge 1}$ and the location of that maximum, respectively. In the i.i.d.\ case, $\gamma = 0$, the random variable $\tau_n $ is uniformly distributed on $\{1,\ldots,n\}$.
Since  the weights $i^\gamma$ increase to infinity, one would expect that $\tau_n/n\to 1$ as $n \to \infty$. The following proposition makes this precise.

\begin{proposition}
\label{prop:loc} 
For $M_n$ and $\tau_n$ as in \eqref{eq:MnX}, we have
\[ 
  \left( 
    \frac{M_n}{n^\gamma} - \beta_n,
    \frac{n-\tau_n}{n}\log n 
  \right)
  \convdistr 
  \left(G, E_\gamma\right)
  \qquad \text{ as }n\to \infty,
\]
where $\beta_n = \log(n/\gamma) - \log \log(n)$ and where $G,E_\gamma$ are independent random variables: $G$ is Gumbel and $E_\gamma$ is exponential with rate~$\gamma$.
\end{proposition}

We start by proving a lemma which is basic for the proof of Proposition~\ref{prop:loc} and the associated point process result given in Proposition~\ref{prop:point_discrete}.

\begin{lemma}
\label{lem:Mz1z2}
For every $z\in\mathbb R$, we have
\[
  \frac{M_{\lfloor n(1 -z / \log n) \rfloor}}{n^\gamma} - \beta_n 
  \convdistr G -\gamma z,
  \qquad \text{ as }n \to \infty,
\]
where $G$ is a Gumbel random variable.
\end{lemma}


\begin{proof}
Letting $M_n(z)=M_{\lfloor n(1 -z / \log n) \rfloor}$ we find from Proposition~\ref{prop:max} that
\begin{align*}
  G_n =
  \frac{M_n(z)}{\lfloor n - n z / \log n \rfloor^\gamma}
  - \beta_{\lfloor n - n z / \log n \rfloor}
  \convdistr G, \qquad \text{ as }n \to \infty.
\end{align*}
Further,
\[
  \frac{M_n(z)}{n^\gamma} - \beta_n \\
  =
  \left(
  G_n + \beta_{\lfloor n - n z / \log n \rfloor}
  \right) \frac{\lfloor n - n z / \log n \rfloor^\gamma}{n^\gamma}
  - \beta_n.
\]
An elementary calculation yields
\begin{equation}\label{eq:toextend}
  \beta_{\lfloor n - n z / \log n \rfloor} \frac{\lfloor n - n z / \log n \rfloor^\gamma}{n^\gamma} - \beta_n
  \to
  - \gamma z,
  \qquad \text{ as }n \to \infty.
\end{equation}
The result follows by Slutsky's lemma and the fact that $\lfloor n - n z / \log n \rfloor^\gamma \sim n^\gamma$, where $a_n \sim b_n$ means that $a_n / b_n \to 1$ as $n \to \infty$.
\end{proof}

The following result establishing convergence of the underlying point processes is close in spirit to, e.g.,~\cite[Thm.\ 1]{weissman75}, and it serves as the basis for Theorem~\ref{thm:point_proc}.
\begin{proposition}\label{prop:point_discrete}
The point process $\hat \xi_n$ on $(-\infty,\infty]^2$ consisting of the points \[\left(\frac{i^\gamma X_i}{n^\gamma}-\beta_n,\frac{n-i}{n}\log n\right)\qquad i=1,2,\ldots\]
converges in distribution as $n\rightarrow\infty$ to the Poisson point process $\hat \xi$ with mean measure
\[\hat \mu(\D x,\D z)=\e^{-x}\D x\times\gamma\e^{-\gamma z}\D z.\]
\end{proposition}
\begin{proof}
Let $Y_{n,i}=i^\gamma X_i/n^\gamma -\beta_n$.
According to the result of Grigelionis, see e.g.~\cite[Thm.\ 16.18]{kallenberg}, applied to a null array of single points, it is only required to show that
\begin{align*}
\sup_{i\geq 1}\,\Prob\{(Y_{n,i},(1-i/n)\log n)\in B\}&\rightarrow 0,\\
\sum_{i\geq 1}\,\Prob\{(Y_{n,i},(1-i/n)\log n)\in B\}&\rightarrow \hat\mu(B),
\end{align*}
for any finite union~$B$ of rectangles in~$(\infty,\infty]^2$. 
In our setting it is sufficient to check the above limits for~$B=[x,\infty]\times [z,\infty]$. The first limit result follows from the monotonicity of $(i/n)^\gamma$ and 
\begin{equation*}\Prob(Y_{n,\lfloor n(1-z/\log n)\rfloor}\geq x)=\exp\{-(x+\beta_n)(1+o(1))\}\rightarrow 0.\end{equation*} 
Using this and Lemma~\ref{lem:Mz1z2} we also find that 
\begin{align*}
&\sum_{\substack{i \geq  1\\ (1 - i/n)\log n\geq z}}\Prob(Y_{n,i}\geq x)=
-(1+\oh(1))\log\prod_{\substack{i \geq  1\\ (1 - i/n)\log n\geq z}}\Prob(Y_{n,i}< x)\\
&=-(1+\oh(1))\log \Prob(M_{\lfloor n-nz/\log n\rfloor}/n^\gamma-\beta_n< x)\rightarrow -\log \Prob(G-\gamma z< x)\\
&=\e^{-x-\gamma z}=\hat\mu([x,\infty]\times [z,\infty]),
\end{align*}
as required.
\end{proof}

\begin{proof}[Proof of Proposition~\ref{prop:loc}]
It follows by the continuous mapping theorem applied to Proposition~\ref{prop:point_discrete} in the same way as Theorem~\ref{thm:main} follows from Theorem~\ref{thm:point_proc}.

Alternatively, one may proceed directly by identifying the limit distribution: 
\[\max\{i^\gamma X_i: {i\in\mathbb N}, n(1-z/\log n)< i\leq n\}/n^\gamma-\beta_n\convdistr G+\log(1-\e^{-\gamma z}),\]
and then expressing the distribution of interest using $M_{\lfloor n-nz/\log n\rfloor}/n^\gamma-\beta_n$ and the above quantity. 
\end{proof}

\section{Gaps of an inhomogeneous Poisson process}
\label{sec:gaps}

Let $0 < T_1 < T_2 < \ldots$ be the points of a Poisson process $\mathcal{N} = (\mathcal{N}_t)_{t \ge 0}$ with rate $\lambda_t = \alpha \, t^{\alpha - 1}$ for some $0 < \alpha < 1$ and with cumulative rate function $\Lambda(t) = \int_0^t \lambda_s \, \D s = t^\alpha$. Note that we assume that $\lambda_1 = \alpha$; the case $\lambda_t = \lambda_1 t^{\a - 1}$ for general $\lambda_1 > 0$ follows by the time change argument, but see also Section~\ref{sec:reg_var}. Recall that $R_i = T_i - T_{i-1}$ for integer $i \ge 1$, where $T_0 = 0$.

Define $T'_i = T_i^\a$ for integer $i \ge 0$, so that $0 < T'_1 < T'_2 < \ldots$ are the points of a unit-rate homogeneous Poisson process $({\mathcal N}'_t)_{t \ge 0}$. Let $X_i = T'_i - T'_{i-1}$ be its interarrival times, for integer $i \ge 1$. The random variables $X_1, X_2, \ldots$ are independent unit exponentials. Put 
\[ 
  \gamma = (1 - \a)/\a \in (0, \infty).
\] 

The following result provides the basic approximation.
\begin{lemma}
\label{lem:RbyX}
We have as $i\to\infty$ that
\begin{equation}
\label{eq:X2R}
  \left|\frac{T_i}{i^{1/\a}}-1\right|\vee \left|\frac{\a R_i}{i^\gamma X_i}-1\right|=\oh(1/\log i) \quad\text{a.s.}
\end{equation}
\end{lemma}

\begin{proof}
Since $(T_i')_i$ is the partial sum process of a sequence of independent unit exponentials, the law of the iterated logarithm states that
\begin{equation*}
  \limsup_{i \to \infty} \frac{T_i'/i - 1}{\sqrt{i^{-1} \log\log i}} = \sqrt{2}
  \qquad
  \text{a.s.},
\end{equation*}
which further implies 
\begin{equation}\label{eq:itlog}
|T_i'/i-1|\log i\rightarrow 0\quad \text{a.s.}
\end{equation}
But then
\[
T_i/i^{1/\a}-1=(T_i'/i)^{1/\a}-1=(T_i'/i-1)(1+\oh(1))/\a \quad \text{a.s.}
\]
and so $|T_i/i^{1/\a}-1|\log i\rightarrow 0$ a.s.\ as required.

Concerning the second part, we write using the mean-value theorem 
\begin{align}
\label{eq:XT2R}
  R_i 
  = T_i - T_{i-1}
  = (T_i')^{1/\a} - (T_{i-1}')^{1/\a}
  = (T_i')^{1/\a} - (T_i' - X_i)^{1/\a}= \a^{-1} \theta_i^{\gamma} X_i
\end{align}
with $T_{i-1}' < \theta_i < T_i'$.
Hence it is left to show that $|(\theta_i/i)^\gamma-1|\log i\rightarrow 0$ a.s., which again follows from~\eqref{eq:itlog}.
\end{proof}

In the following we relate the points of the point process $\xi_t$ in Theorem~\ref{thm:point_proc} to the corresponding points of the process $\hat\xi_{\lceil t^\a\rceil}$ in Proposition~\ref{prop:point_discrete} with rescaled second component.

\begin{lemma}\label{lem:points_close}
Let $B=[x_1,x_2]\times[z_1,z_2]$ and put
\begin{align*}
  u_i(t)
  &= (\lambda_t R_i-b_t,(1-T_i/t)\log t),\\
  v_i(t)
  &= (i^\gamma X_i/n^\gamma-\beta_n,(1-i/n)\log(n)/\a^2)
\end{align*}
with $n=n(t)=\lceil t^\a\rceil$.
Then 
\[
  \sup_i\{\|u_i(t)-v_i(t)\|_1:v_i(t)\in B\text{ or }u_i(t)\in B\}
  \rightarrow 0\quad \text{a.s.}
\]
as $t\rightarrow \infty$ with the convention that~$\sup\varnothing = 0$.
\end{lemma}

\begin{proof}
Letting $I_v(t)=\{i\geq 1:v_i(t)\in B\}$ we see that $i/n\to 1$ and hence also $i/t^\a\to 1$ uniformly in $i\in I_v(t)$ as $t\to\infty$.
Now according to Lemma~\ref{lem:RbyX}, for all $i\in I_v(t)$, we have
\begin{equation}\label{eq:approx_main}\a R_i=i^\gamma X_i(1+\eta'_i),\qquad T_i=i^{1/\a}(1+\eta''_i)\end{equation}
where $|\eta'_i|\vee|\eta''_i|=\oh(1/\log t)$ as $t\rightarrow \infty$ a.s.
So we have a.s.
\begin{align}
\label{lim1}\lambda_t R_i-b_t&\leq i^\gamma X_i/t^{\a\gamma}(1+\oh(1/\log t))-b_t\\
&=(i^\gamma X_i/n^\gamma-\beta_n)(1+\oh(1/\log t))+\oh(1),\nonumber
\end{align}
where in the last line we used the facts: $\lceil t^\a\rceil^\gamma/t^{\a\gamma}=1+\oh(1/\log t)$ and $b_t=\beta_{t^\a}+\oh(1)=\beta_n+\oh(1)$. This and the analogous lower bound imply that
\[\sup_{i\in I_v(t)}|(\lambda_t R_i-b_t)-(i^\gamma X_i/n^\gamma-\beta_n)|\to 0\]
as $t\to\infty$ a.s., because $|i^\gamma X_i/n^\gamma-\beta_n|$ is bounded for the indices of interest.
Upon recalling that $i/t^\a-1\to 0$ uniformly in $i\in I_v(t)$, for all such $i$ we find that
\begin{align}
\label{lim2}
\a^2(1-T_i/t)\log t&\leq\a^2(1-i^{1/\a}(1+\oh(1/\log t))/t)\log t\\
&=\a(1-(i/t^\a)^{1/\a})\log (t^\a)+\oh(1) \nonumber\\
&=(1-i/t^\a)\log (t^\a)(1+\oh(1))+\oh(1)\nonumber\\
&=(1-i/n)\log(n)(1+\oh(1))+\oh(1).\nonumber
\end{align}
This and the analogous lower bound yield 
\[\sup_{i\in I_v(t)}|\a^2(1-T_i/t)\log t-(1-i/n)\log n|\to 0\]
as $t\to\infty$ a.s., because now $|1 - i/n| \log n$ is bounded for the indices of interest.

Next, consider the set of indices $I_u(t)=\{i\geq 1:u_i(t)\in B\}$. In this case we use the fact that $T_i/t\to 1$ uniformly in $i\in I_u(t)$. 
Furthermore, with probability~1 as $t\rightarrow \infty$ the corresponding indices~$i$ converge to $\infty$ too, and since $T_i\sim i^{1/\a}$ we must have that $i/t^\a\to 1$ uniformly in $i\in I_u(t)$. Thus~\eqref{lim1} holds true and hence also 
\begin{equation}\label{lim11}i^\gamma X_i/n^\gamma-\beta_n\geq (\lambda_t R_i-b_t)(1+\oh(1/\log t))+\oh(1).\end{equation}
The corresponding upper bound, as well as the bounds on $(1-i/n)\log(n)/\a^2$ stemming from~\eqref{lim2}, complete the proof, because $|\lambda_t R_i-b_t|$ and $|(1-T_i/t)\log t|$ are bounded for all $i\in I_u(t)$.
\end{proof}

\begin{remark}
\label{rem:rescaling}\rm 
The point process $\sum_i\delta_{v_i(n)}$ with $v_i(n)$ defined in Lemma~\ref{lem:points_close} 
is a rescaled version of $\hat\xi_n$ in Proposition~\ref{prop:point_discrete}, and the proof of the latter easily yields that $\sum_i\delta_{v_i(n)}$ converges in distribution to a Poisson point process with intensity measure for the set~$[x,\infty]\times[z,\infty]$ given by 
\[
  \hat\mu([x,\infty]\times[\a^2 z,\infty])
  =\e^{-x-\a(1-\a)z}
  =\mu([x,\infty]\times[z,\infty]).
\]
That is, the corresponding limit is~$\xi$.
\end{remark}

The following lemma shows that compact sets of the form $[x,\infty]\times[z,\infty]$ can be truncated to finite rectangles.

\begin{lemma}\label{lem:bounding}
For any $\epsilon>0$ and $z,x<\infty$ there exist $z'>z$ and $x'>x$ such that 
\[
  \limsup_{t\rightarrow \infty}\,
  \Prob(\xi_t(([x,\infty]\times[z,\infty]) \setminus ([x,x']\times[z,z']))>0)<\epsilon. 
\]
\end{lemma}

\begin{proof}
Put $n=\lceil t^\a\rceil$ and observe using~\eqref{eq:XT2R} that
\begin{align*}
  \max_{i\leq n/2} \a R_i\leq \max_{i \leq n/2} (T_i')^\gamma X_i \leq (T'_{\lceil n/2\rceil})^\gamma\max_{i \leq n/2}X_i=(n/2)^\gamma\log n(1+\oh_p(1)),
\end{align*}
where in the last equality we used the law of large numbers applied to $T_i'$ and the fact that $\max_{i=1,\ldots,k} X_i - \log(k)$ is asymptotically Gumbel. But then 
 \[
  \lambda_t \max_{i\leq n/2}R_i-b_t
  \leq 2^{-\gamma}\a\log(t)(1+\oh_p(1))-b_t\rightarrow -\infty 
  \]
in probability.
Thus it is sufficient to restrict our attention to the indices $i>n/2$, in which case we have~\eqref{eq:approx_main} for all such $i$ a.s.
 
Observe that for $i>n/2$ the bound~\eqref{lim11} is still true.
Letting $I(t,z)$ be the set of indices $i$ such that $(1-T_i/t)\log t\in[z,z+1]$ or $(1-i/n)\log(n)/\a^2\in[z,z+1]$, we see from the proof of Lemma~\ref{lem:points_close} that $i/t^\a-1\to 0$ uniformly in $i\in I(t,z)$, and also that
\[
  \sup_{i\in I(t,z)}
  \lvert (1-T_i/t)\log t-(1-i/n)\log(n)/\a^2\rvert
  \to 0
  \qquad\text{as }t\to\infty\qquad\text{ a.s.}
\]
Hence for any fixed $\delta>0$ with arbitrarily high probability the following is true for large enough $t$: if for some $i>n/2$ it is true that \[\lambda_t R_i-b_t\geq x\text{ and } (1-T_i/t)\log t\geq z\] then
\[
  i^\gamma X_i/n^\gamma-\beta_n\geq x-\delta
  \text{ and }
  (1-i/n)\log(n)/\a^2\geq z-\delta,
\]
because for $i\notin I(t,z)$ the monotonicity of $T_i$ implies $(1-i/n)\log(n)/\a^2>z+1$.
Thus it is left to apply Proposition~\ref{prop:point_discrete} and to note that $\hat\mu(B_1),\hat\mu(B_2)\rightarrow 0$ with 
\[B_1=[x-\delta,\infty]\times[\a^2(z'-\delta),\infty],\quad B_2=[x'-\delta,\infty]\times[\a^2(z-\delta),\infty]\] as $x',z'\rightarrow \infty$, which implies that $\Prob(\hat\xi(B_i)>0)\rightarrow 0$. 
\end{proof}

\begin{proof}[Proof of Theorem~\ref{thm:point_proc}]
According to~\cite[Thm.\ 1]{kallenberg_SPA} it is sufficient to show that
\begin{align}
\label{thm:point_proc:a}
\lim_{t\rightarrow\infty}\Prob(\xi_t(B)=0)&= \Prob(\xi(B)=0), \\ \label{thm:point_proc:b}
\limsup_{t\rightarrow\infty}\Prob(\xi_t(K)>1) &\leq \Prob(\xi(K)>1),
\end{align}
where $K$ is a compact rectangle in~$(-\infty,\infty]^2$ and $B$ is a finite union of such rectangles.
According to Lemma~\ref{lem:bounding} we may choose $x_1<x_2,z_1<z_2$ such that
\begin{align*}
  0
  \leq \Prob(\xi_t(B')=0)-\Prob(\xi_t(B)=0)
  \leq \Prob \bigl( \xi_t(B \setminus ([x_1,x_2]\times[z_1,z_2]))>0 \bigr)
  < \epsilon
\end{align*}
for $B'=B\cap ([x_1,x_2]\times [z_1,z_2])$ and all $t$ large enough.  Furthermore, we may additionally ensure that $0\leq \Prob(\xi(B')=0)-\Prob(\xi(B)=0)\leq \epsilon$.
A similar observation holds true with respect to $\Prob(\xi_t(B)>1)-\Prob(\xi_t(B')>1)$ and the corresponding difference for the process $\xi$. Hence it is sufficient to prove~\eqref{thm:point_proc:a} and~\eqref{thm:point_proc:b} for any finite rectangle $K$ and a finite union $B$ of such rectangles.

Fix $\delta>0$ and define the $\delta$-enlarged set $B_{\delta+}=\{v:d(v,B)<\delta\}$ and $\delta$-narrowed set $B_{\delta-}=\{v:d(v,{B}^c)>\delta\}$, where $d$ is the Euclidean distance. According to Lemma~\ref{lem:points_close} with the respective rectangle chosen to cover $B$, we have
\begin{align*}
  \Prob(\#\{i: v_i(t)\notin B_{\delta+}\}=0)-\epsilon
  &\leq \Prob(\xi_t(B)=0) \\
  &\leq \Prob(\#\{i: v_i(t)\notin B_{\delta-}\}=0)+\epsilon
\end{align*} 
for all $t$ large. Noting that $\mu(\partial B)=0$ we obtain \eqref{thm:point_proc:a} from Remark~\ref{rem:rescaling} based on Proposition~\ref{prop:point_discrete}.
In a similar way we also find that $\Prob(\xi_t(K)>1)\to\Prob(\xi(K)>1)$. The proof is complete.
\end{proof}

\section{Extensions to regular variation}
\label{sec:reg_var}

Let $\RV_\rho$ denote the set of measurable functions $f : \R_+ \to \R_+$ which are regularly varying at $\infty$ with index~$\rho\in\R$, i.e., satisfying $f(ut)/f(t)\to u^\rho$ as $t\to \infty$ for all $u>0$. Any such $f$ can be represented as $f(t)=t^\rho\ell(t)$ with $\ell\in \RV_0$ a slowly varying function. Regularly varying functions are thus a generalization of the power functions considered above. Let us also recall the basic theorem concerning regularly varying functions $f\in \RV_\rho$, the Uniform Convergence Theorem~\cite[Thm.\ 1.5.2]{bingham1989regular}:
\begin{equation}\label{UCT}\tag{UCT}
f(ut)/f(t)\to u^\rho,\qquad\text{uniformly in }u
\end{equation}
on intervals $[a,b]$ with $0<a\leq b<\infty$ for $\rho\leq 0$,
and on intervals $(0,b]$ for $\rho>0$ if $f$ is locally bounded.

Assume that the rate function $t \mapsto \lambda_t$ is in $\RV_{\alpha-1}$ for some $\alpha \in (0, 1)$, so that $\Lambda \in \RV_\alpha$. Let $V(t) = \Lambda^{-1}(t)$ be the inverse function of $\Lambda$ and let $v(t) = \D V(t) / \D t$ be its derivative. Then $V \in \RV_{1/\alpha}$ and $v = 1 / (\lambda \circ V) \in \RV_\gamma$ with $\gamma = (1-\alpha)/\alpha$.
The point process $\mathcal{N}'_t = \mathcal{N}_{V(t)}$ is a unit-rate homogenous Poisson process with epochs $T_i' = \Lambda(T_i)$. 

We generalize our main results imposing just one condition on the slowly varying function $\ell$ associated with $v$, i.e., $v(t) = t^\gamma \ell(t)$; see Condition~\ref{cond:loc} below.
The basis of our analysis will be the approximation $R_i = V(T_i') - V(T_{i-1}') \approx v(i) X_i$ inspired by the mean-value theorem applied to $V$ and the strong law of large numbers applied to the partial sum sequence $T_i'$. Therefore, we first study the behaviour of the maximum of the weighted exponentials $v(i) X_i$.

\subsection{Weighted exponentials}

Let $X_1, X_2, \ldots$ be a sequence of i.i.d.\ unit exponentials and let $v \in \RV_\gamma$ with $\gamma > 0$. Write $v(t) = t^\gamma \ell(t)$ with $\ell \in \RV_0$. We consider the maximum of the weighted exponentials $v(i) X_i$ and the location of that maximum:
\begin{equation*}
  M_n^* = \max_{i=1,\ldots,n} \{ v(i) X_i \}
  \qquad\text{and}\qquad
  \tau_n^* = \min \{ i = 1, \ldots, n : v(i) X_i = M_n^* \}.
\end{equation*}

\begin{lemma}
\label{lem:tauto1}
We have \[\frac{M^*_n}{v(n)} \, \Big/ \, \frac{M_n}{n^\gamma}=1 + \oh_p(1),\qquad \frac{\tau_n^*}{n} = 1 + \oh_p(1)\] as $n \to \infty$.
\end{lemma}

\begin{proof}
Let $0 < h < 1$. First, we prove that $\Prob( \tau_n^* / n > h ) \to 1$ as $n \to \infty$, or equivalently, that 
\begin{equation} 
\label{eq:tauto1}
  \lim_{n \to \infty} \Prob( M_{\floor{nh}}^* < M_n^* ) = 1.
\end{equation}

On the one hand, we have
\[
  \frac{M_{\floor{nh}}^*}{v(n)}
  =
  \max_{i=1,\ldots,{\floor{nh}}} \frac{v(i)}{v(n)} X_i
  \le
  \max_{i=1,\ldots,{\floor{nh}}} \frac{v(i)}{v(n)}
  \max_{i=1,\ldots,{\floor{nh}}} X_i.
\]
But $v\in\RV_\gamma$ can be assumed to be locally bounded [otherwise redefine $v$ by $v(t) = v(\floor{t})$], and so by \eqref{UCT} it follows that
\[
  \lim_{n \to \infty} \max_{i=1,\ldots,{\floor{nh}}} \frac{v(i)}{v(n)}
  =
  \max_{u \in [0, h]} u^\gamma
  =
  h^\gamma.
\]
Since $X_1, X_2, \ldots$ are i.i.d.\ unit exponentials, we have $\max_{i = 1, \ldots, \floor{nh}} X_i = \log \floor{nh} + \Oh_p(1) = \log(n) \{ 1 + \oh_p(1) \}$ and thus
\[
  \frac{M_{\floor{nh}}^*}{v(n)}
  \le
  h^\gamma \log(n) \{1 + \oh_p(1)\}
  \qquad\text{ as }n \to \infty.
\]

On the other hand, let $g \in (h, 1)$. By a similar argument as in the previous paragraph, we find
\begin{align*}
  \frac{M_n^*}{v(n)}
  \ge
  \max_{i=\floor{ng},\ldots,n} \frac{v(i)}{v(n)} X_i
  &\ge
  \min_{i=\floor{ng},\ldots,n} \frac{v(i)}{v(n)}
  \max_{i=\floor{ng},\ldots,n} X_i \\
  &=
  g^\gamma \log(n) \{1 + \oh_p(1)\}
  \qquad\text{ as }n \to \infty.
\end{align*}
Since $h^\gamma < g^\gamma$, we obtain \eqref{eq:tauto1}, as required.

Concerning the first statement, observe from above and Proposition~\ref{prop:loc} that with arbitrarily high probability 
\[M^*_n=\max_{i=\floor{ng},\ldots,n} \ell(i)i^\gamma X_i,\qquad M_n=\max_{i=\floor{ng},\ldots,n} i^\gamma X_i\]
for large enough~$n$. Hence it is sufficient to show that 
\begin{equation}\label{eq:ell_lim}\max_{i=\floor{ng},\ldots,n} \ell(i)/\ell(n)-1\rightarrow 0\qquad\text{ as }n\rightarrow\infty\end{equation}
and the same for $\min$, which again follows from~\eqref{UCT}. 
\end{proof}

In order to generalize Proposition~\ref{prop:max} we need a stronger statement than the readily available~\eqref{eq:ell_lim}, and so we assume the following additional condition on the slowly varying function~$\ell$.
\begin{condition}
\label{cond:loc}
Whenever $0 < \epsilon(t) \to 0$ as $t \to \infty$, we have
\[
  \log(t)  \left( \frac{\ell([1+\epsilon(t)]t)}{\ell(t)} - 1 \right)
  \to 0.
\]
\end{condition}
In Section~\ref{sec:condition} we provide a simple sufficient criterion under which Condition~\ref{cond:loc} holds.
It is important to realize that Condition~\ref{cond:loc} is equivalent to a seemingly stronger condition stated in the following lemma.
\begin{lemma}\label{lem:cond_equiv}
Condition~\ref{cond:loc} is equivalent to
\begin{equation}\label{eq:cond_equiv}
  \log(t)  \sup_{-\epsilon(t)\leq x\leq \epsilon(t)}\left| \frac{\ell((1+x)t)}{\ell(t)} - 1 \right|
  \to 0
\end{equation}
for any $0<\epsilon(t)\rightarrow 0$.
\end{lemma}
\begin{proof}
Given in Appendix~\ref{appendix}.
\end{proof}

Recall $\beta_n = \log(n/\gamma) - \log\log(n)$.

\begin{lemma}
\label{lem:Mnstarz1z2}
Assuming Condition~\ref{cond:loc} we have $M_n^*/v(n)-\beta_n\convdistr G$.
\end{lemma}

\begin{proof}
Since $(\tau_n^* \wedge \tau_n) / n = 1 + \oh_p(1)$ as $n \to \infty$ by Lemma~\ref{lem:tauto1} and Proposition~\ref{prop:loc}, we can find $\epsilon_n > 0$ such that $\epsilon_n \to 0$ and $\Prob[ \tau_n^* \wedge \tau_n > n(1 - \epsilon_n)] \to 1$ as $n \to \infty$. 
Hence with arbitrarily high probability we have
\begin{align*}
  \left\lvert
    \frac{M_n^*}{v(n)} - \frac{M_n}{n^\gamma}
  \right\rvert
  &=
  \left\lvert
    \max_{n(1-\epsilon_n)<i\leq n}
    \left\{ \frac{i^\gamma \ell(i)}{n^\gamma \ell(n)} X_i \right\}
    -
    \max_{n(1-\epsilon_n)<i\leq n}
    \left\{ \frac{i^\gamma}{n^\gamma} X_i \right\}
  \right\rvert
  \\
  &\le
  \max_{n(1-\epsilon_n)<i\leq n}
  \left\{
    \left\lvert
      \frac{\ell(i)}{\ell(n)} - 1
    \right\rvert
    \frac{i^\gamma}{n^\gamma} X_i
  \right\}
  \le
  \frac{M_n}{n^\gamma}
  \sup_{-\epsilon_n<x\leq 0}
  \left\lvert
    \frac{\ell(n(1+x))}{\ell(n)} - 1
  \right\rvert
  .
\end{align*}
Lemma~\ref{lem:cond_equiv} and Lemma~\ref{lem:Mz1z2} show that $M_n^*/v(n)=M_n/n^\gamma+\oh_p(1)$ completing the proof.
\end{proof}

\begin{proposition}
Let $v \in \RV_\gamma$ for some $\gamma > 0$ and put $\ell(t) = t^{-\gamma} v(t)$. If $\ell$ satisfies Condition~\ref{cond:loc} then Proposition~\ref{prop:loc} and Proposition~\ref{prop:point_discrete} hold with $M^*_n, \tau^*_n, v(i),v(n)$ in place of $M_n,\tau_n,i^\gamma,n^\gamma$.
\end{proposition}
\begin{proof}One may follow the same steps as in the original proofs. In addition, for the analogue of~\eqref{eq:toextend} we use Lemma~\ref{lem:RV}, whereas the extension of Proposition~\ref{prop:point_discrete} requires showing that
\[\sup_{1\leq i\leq n(1-z/\log n)}\exp(-(x+\beta_n)v(n)/v(i))\to 0,\]
which follows from \eqref{UCT} applied to the function $v(\lfloor t\rfloor)$. 
\end{proof}


\subsection{Gaps of a Poisson process}


Let $( \mathcal{N}_t )_{t \ge 0}$ be an inhomogenous Poisson process as in the beginning of this section. 
%
As a consequence of Lemma~\ref{lem:cond_equiv}, we have that\begin{multline}
\label{eq:vloc}
  0 < \delta(t) = \oh \left( 1 / \log t \right)
  \text{ as $t \to \infty$ implies } \\
  \lim_{t \to \infty} \log(t)
  \sup_{-\delta(t) \le x \le \delta(t)}
  \left\lvert \frac{v((1+x)t)}{v(t)} - 1 \right\rvert
  = 0,  
\end{multline}
because $(1\pm\delta(t))^\gamma-1=\oh(1/\log t)$.

Let us now provide a generalization of Lemma~\ref{lem:RbyX}.
\begin{lemma}
\label{lem:approxRbyX}
If $\ell$ satisfies Condition~\ref{cond:loc}, then
\begin{equation*}
  \left|\frac{T_i}{V(i)}-1\right|\vee \left|\frac{R_i}{v(i)X_i}-1\right|=\oh(1/\log i) \quad\text{a.s.}
\end{equation*}
as $i\rightarrow\infty$.
\end{lemma}

\begin{proof}
From the monotonicity of $V$ and from~\eqref{eq:itlog}, we find that a.s.
\[T_i/V(i)-1=V(T'_i)/V(i)-1\leq V(i(1+\oh(1/\log i)))/V(i)-1,\]
which is $\oh(1/\log i)$ by Lemma~\ref{lem:RV}. A similar bound from below completes the proof of the first part.

For the second part we write
\begin{equation}\label{eq:R_RV}
  R_i 
  = T_i - T_{i-1}
  = V(T_i') - V(T_{i-1}')
  = \int_{T_{i-1}'}^{T_i'} v(t) \, \D t,
\end{equation}
so that
\[
  R_i - v(i) X_i
  =
  \int_{T_{i-1}'}^{T_i'} \{ v(t) - v(i) \} \, \D t.
\]
But then
\[
  \left\lvert \frac{R_i}{v(i) X_i} - 1 \right\rvert
  \le
  \frac{1}{X_i}
  \int_{T_{i-1}'}^{T_i'}
    \left\lvert
      \frac{v(t)}{v(i)} - 1
    \right\rvert \,
  \D t
  \le
  \sup_{T_{i-1}'/i-1 \le x \le T'_i/i-1}
    \left\lvert
      \frac{v(i(1+x))}{v(i)} - 1
    \right\rvert.  
\]
From~\eqref{eq:itlog} and \eqref{eq:vloc} we find that the last term is $\oh(1/\log i)$ a.s.\ as required.
\end{proof}

\begin{theorem}\label{thm:RV}
If the rate function $\lambda\in \RV_{\a-1}$, where $\a\in(0,1)$, is such that $\ell$ satisfies Condition~\ref{cond:loc}, then Theorem~\ref{thm:main} and Theorem~\ref{thm:point_proc} hold with such $\lambda_t$ and 
\[b_t=\log\Lambda(t)-\log\log t-\log(1-\a).\]
\end{theorem}

\begin{proof}
In this more general setting we use $n=\lceil\Lambda(t)\rceil$ and so $\log n\sim\a\log t$. 
Concerning the generalization of Lemma~\ref{lem:points_close} we only need to show that~\eqref{lim1} and~\eqref{lim2} hold when adapted according to Lemma~\ref{lem:approxRbyX}.
That is,
\begin{align*}
&\lambda_t v(i)X_i(1+\oh(1/\log t))-b_t=(v(i)X_i/v(n)-\beta_n)(1+\oh(1/\log t))+\oh(1),\\
&\a^2\log(t)(1-V(i)/t)=\log(n)(1-i/n)(1+\oh(1))+\oh(1).
\end{align*}
This hinges on the following: (i) $\lambda_t v(n)=1+\oh(1/\log t)$, (ii) $b_t=\beta_n+\oh(1)$, and (iii) $\a(1-V(i)/t)=(1-i/n)(1+\oh(1))+\oh(1/\log t)$ uniformly in $i\in I_v(t)$. Identity (i) holds, because
by \eqref{eq:vloc}
\begin{equation}\label{eq:lambdaRV}
  \lambda_t \, v(\ceil{\Lambda(t)})
  =
  \frac{v(\ceil{\Lambda(t)})}{v(\Lambda(t))}
  =
  1 + \oh(1/\log t),
  \qquad \text{ as }t \to \infty,
\end{equation}
where, indeed, $\abs{ \ceil{\Lambda(t)} / \Lambda(t) - 1 } < 1/\Lambda(t) = \oh(1 / \log t)$. Identity (ii) is rather obvious, whereas concerning (iii) we have
\[\a(1-V(i)/t)=\a(1-V(i)/V(\Lambda(t)))=\a(1-V(i)/V(n))(1+\oh(1))+\oh(1/\log t),\]
but $1-V(n(1+i/n-1))/V(n)=(1-i/n)(1+\oh(1))/\a$
by a slight extension of Lemma~\ref{lem:RV} upon noting that $i/n-1=\oh(1)$ uniformly in $i$ concerned.

It is left to show that Lemma~\ref{lem:bounding} still holds, and the only non-trivial step is to show that
\begin{equation}\label{eq:infty}
\lambda_t \max_{i\leq n/2} R_i-b_t\rightarrow-\infty
\end{equation}
in probability and hence a.s., which we obtain in the following.
Observe from~\eqref{eq:R_RV} that 
\begin{align*}
\max_{i\leq n/2}R_i\leq \max_{i\leq n/2}X_i\sup_{T_{i-1}'\leq t\leq T_i'}v(t)\leq \max_{i\leq n/2}X_i\sup_{t\leq T_{\lceil n/2\rceil}'}v(t),
\end{align*}
where $\max_{i\leq n/2}X_i=\log n(1+\Oh_p(1))$ and concerning the latter term we have
\[\sup_{t\leq T_{\lceil n/2\rceil}'}v(t)/v(n)=\sup_{t\leq n/2(1+\oh(1))}\frac{v(t)}{v(n)}\rightarrow 2^{-\gamma}\quad \text{a.s.}\]
by \eqref{UCT}, provided that $v$ is locally bounded. This shows that $\max_{i\leq n/2}R_i/v(n)\leq 2^{-\gamma}\log n(1+\Oh_p(1))$ and hence~\eqref{eq:infty} holds in view of~\eqref{eq:lambdaRV}.
In general, however, we only have that $v$ is bounded on $[a,b]$ for some $a$ and all~$b$. With arbitrarily high probability we may choose an index $j$ such that $T_j'>a$, and then the above steps can be repeated for $\max_{j<i\leq n/2}R_i$, whereas obviously $T_j/v(n)\rightarrow 0$ a.s.
\end{proof}

\subsection{Comments on the assumed condition}
\label{sec:condition}

Let us note that virtually all standard examples of slowly varying functions,  e.g.\ $\log^u t,u\in\R$ and $\log\log t$, satisfy Condition~\ref{cond:loc}. This can be easily checked using the following result.

\begin{lemma}
\label{lem:condition}
Condition~\eqref{cond:loc} holds true if $\ell\in \RV_0$ is eventually differentiable and
\begin{equation}
\label{eq:cond_sufficient}
  \frac{t \, \ell'(t)}{\ell(t)}
  =
  \Oh(1/\log t),
  \qquad \text{as } t \to \infty.
\end{equation}
\end{lemma}

\begin{proof}
Using the mean value theorem we have
\[  |\ell([1+\epsilon(t)]t)-\ell(t)|\leq t\epsilon(t)\sup_{t \le s \le [1+\epsilon(t)]t}
  |\ell'(s)|.\]
Moreover, 
\[
  \sup_{t \le s \le [1+\epsilon(t)]t} |\ell'(s)|
  \leq 
  \sup_{t \le s \le [1+\epsilon(t)]t}
    \left|\frac{s \, \ell'(s)\log s}{\ell(s)}\right|
  \sup_{t \le s \le [1+\epsilon(t)]t}
    \left|\frac{\ell(s)}{s\log s}\right|,
\]
where the first term on the right hand side is $\Oh(1)$ according to~\eqref{eq:cond_sufficient}.
Hence Condition~\eqref{cond:loc} holds if
\[\sup_{t \le s \le [1+\epsilon(t)]t}\frac{t\log(t)\ell(s)}{s\log(s)\ell(t)}\]
is bounded for large $t$, but this term tends to~1 by \eqref{UCT} applied to the regularly varying function $(t\log(t))^{-1}\ell(t)$.
\end{proof}

Concerning Theorem~\ref{thm:RV} it is more useful to express the sufficient condition of Lemma~\ref{lem:condition} using the slowly varying function associated with the rate function $\lambda_t$ instead of that associated with~$v(t)$, which is the content of the next result. 

\begin{proposition}\label{prop:cond}
Let $\lambda_t=t^{\a-1}\ell_\lambda(t)$ with $\a\in(0,1)$ and $\ell_\lambda\in\RV_0$. 
If $\ell_\lambda$ is eventually continuously differentiable and if
\[
  \frac{t \, \ell_\lambda'(t)}{\ell_\lambda(t)}=\Oh(1/\log t),
  \qquad \text{as } t \to \infty,
\]
then Condition~\ref{cond:loc} is satisfied and the result of Theorem~\ref{thm:RV} holds true.
\end{proposition}

\begin{proof}
First, we show that~\eqref{eq:cond_sufficient} is equivalent to
\begin{equation}
\label{eq:cond2}
  \frac{\Lambda(t) \, \lambda'_t}{\lambda_t^2}
  =
  -\gamma+\Oh(1/\log t).
\end{equation}
Since $v(t)=t^\gamma \ell(t)$ and $v(t)=1/\lambda_{V(t)}$ we find that
\[
  \frac{t \, \ell'(t)}{\ell(t)}
  =
  -\gamma-t \, \lambda'_{V(t)}/\lambda^2_{V(t)}.
\]
Plugging in $t=\Lambda(t)$ and noting that $\log \Lambda(t)\sim \a\log t$ we confirm the equivalence.

Thus it is sufficient to establish that 
\begin{align*}
  \frac{\lambda'_t \, t}{\lambda_t}
  &=
  \a-1+\Oh(1/\log t), 
  &
  \frac{\Lambda(t)}{\lambda_t t}
  &=
  1/\a+\Oh(1/\log t).
\end{align*}
The left statement is a result of a simple calculation, and so we concentrate on the right statement.
Using integration by parts we find
\[
  \Lambda(t)
  =
  \int_c^t x^{\a-1}\ell_\lambda(x) \, \D x\ 
  =\ 
  \frac{1}{\a} t^\a \, \ell_\lambda(t)
  + \Oh(1)-\frac{1}{\a} \int_c^t x^{\a} \, \ell'_\lambda(x) \, \D x
\]
for all $t>c$ and some level~$c$ (to be fixed high enough). Hence it is left to show that 
\begin{equation}
\label{eq:withKaramata}
  \frac{\int_c^t x^{\a} \, \ell'_\lambda(x) \, \D x}{t^\a \, \ell_\lambda(t)}
  \log t
  = \Oh(1).
\end{equation}
From our assumption we see that $|\ell_\lambda'(x)|\leq C\ell_\lambda(x)/(x\log x)$ for large enough~$x$. Finally, by Karamata's theorem~\cite[Prop.\ 1.5.8]{bingham1989regular}	we have
\[
  \frac%
    {\int_c^t x^{\a} \, C \, \ell_\lambda(x)/(x\log x) \, \D x}%
    {t^\a\,\ell_\lambda(t)/\log t}
  \to C\a,
\]
because $\ell_\lambda(t)/\log t\in RV_0$, and so~\eqref{eq:withKaramata} follows.
\end{proof}

As mentioned above, virtually all standard examples of slowly varying functions satisfy the assumption of Proposition~\ref{prop:cond}. In particular, so do $\ell_\lambda(t)=a\log^u t$ and $\ell_\lambda(t)=a\log\log t$ for $a>0,u\in\mathbb R$. Hence $\lambda_t=a t^{\a-1}\log^u t$ and $\lambda_t=a t^{\a-1}\log\log t$ are examples of rate functions to which the asymptotic results in this work apply. For a simple example that does not satisfy the assumption of Proposition~\ref{prop:cond}, consider $\ell_\lambda(t)=\e^{(\log\log t)^2}=(\log t)^{\log\log t}$. Indeed, this is a slowly varying function for which $\log(t) \, t \, \ell_\lambda'(t)/\ell_\lambda(t)=2\log\log t$ is unbounded.

\appendix
\section{}\label{appendix}

\begin{proof}[Proof of Lemma~\ref{lem:cond_equiv}]
It is clearly sufficient to show that Condition~\ref{cond:loc} implies~\eqref{eq:cond_equiv}. Firstly, from Condition~\ref{cond:loc} we have that
\[\log t\left(\frac{\ell([1-\epsilon(t)]t)}{\ell(t)} - 1\right)=-\log t\left(\frac{\ell([1+\hat\epsilon(\hat t)]\hat t)}{\ell(\hat t)} - 1\right)\frac{\ell([1-\epsilon(t)]t)}{\ell(t)}\to 0,\]
where $\hat t=(1-\epsilon(t))t$ and $\hat\epsilon(\hat t)=\epsilon(t)/(1+\epsilon(t))$. Next, for any $\epsilon>0$ and any large~$t$ we can choose $x(t)\in[-\epsilon(t),\epsilon(t)]$ such that 
\[\left| \frac{\ell([1+x(t)]t)}{\ell(t)} - 1 \right|+\epsilon/\log t>\sup_{-\epsilon(t)\leq x\leq \epsilon(t)}\left| \frac{\ell([1+x]t)}{\ell(t)} - 1 \right|.\]
But the term on the left when multiplied by $\log t$ must converge to $\epsilon$, because $x(t)\rightarrow 0$. The limit result in~\eqref{eq:cond_equiv} follows since $\epsilon>0$ is arbitrary.
\end{proof}

The following technical result concerning regularly varying functions may well exist in the literature.
\begin{lemma}
\label{lem:RV}
Let $f$ be a positive, increasing, and absolutely continuous function such that its Radon--Nikodym derivative $f'$ is in $\RV_{\tau-1}$ for some $\tau > 0$. If $x_t \to 0$ as $t \to \infty$, then
\[
  \frac{f(t(1+x_t))}{f(t)} - 1
  = \tau \, x_t \{ 1 + \oh(1) \},
  \qquad \text{as $t \to \infty$}.
\]
\end{lemma}

\begin{proof}
We have
\begin{align*}
  \frac{f(t(1+x_t)) - f(t)}{t \, f'(t)} - x_t
  &= \int_1^{1+x_t} \left( \frac{f'(zt)}{f'(t)} - 1 \right) \, \D z
\end{align*}
and thus
\[
  \left\lvert \frac{f(t(1+x_t)) - f(t)}{t \, f'(t)} - x_t \right\rvert
  \le \abs{ x_t } \, \sup_{\abs{z-1} \le \abs{x_t}} \left\lvert \frac{f'(zt)}{f'(t)} - 1 \right\rvert.
\]
But then
\begin{align*}
    &\left\lvert \frac{f(t(1+x_t))}{f(t)} - 1 - \tau \, x_t \right\rvert
=
  \left\lvert  \frac{t \, f'(t)}{f(t)} \frac{f(t(1+x_t)) - f(t)}{t \, f'(t)} - \tau \, x_t \right\rvert \\
  &\le
  \frac{t \, f'(t)}{f(t)}
  \left\lvert \frac{f(t(1+x_t)) - f(t)}{t \, f'(t)} - x_t \right\rvert
  +
  \left\lvert \frac{t \, f'(t)}{f(t)} - \tau \right\rvert \abs{x_t} \\
  &\le
  \left\{
    \frac{t \, f'(t)}{f(t)}
    \sup_{\abs{z-1} \le \abs{x_t}} \left\lvert \frac{f'(zt)}{f'(t)} - 1 \right\rvert
    +
    \left\lvert \frac{t \, f'(t)}{f(t)} - \tau \right\rvert
  \right\}
  \abs{x_t}.
\end{align*}
The term in curly brackets converges to zero by 
\eqref{UCT} applied to $f'$, the fact that $\lim_{t \to \infty} x_t = 0$, and the direct half of Karamata's theorem, see e.g.\ Theorem~1.5.11 in \cite{bingham1989regular}.
\end{proof}

\section*{Acknowledgments}

We gratefully acknowledge the comments and suggestions by the anonymous reviewers, who have provided additional references and pointed out various ways to improve the writing.

J.~Ivanovs acknowledges support by T.N.~Thiele Center at Aarhus University.
J.~Segers gratefully acknowledges funding by contract ``Projet d'Act\-ions de Re\-cher\-che Concert\'ees'' No.\ 12/17-045 of the ``Communaut\'e fran\c{c}aise de Belgique'' and by IAP research network Grant P7/06 of the Belgian government (Belgian Science Policy).

\end{document}